%% file: 0main.tex
\documentclass[abstracton, paper=a4, fontsize=11pt,DIV=14,bibliography=totoc, final]{scrartcl}

\pdfoutput=1

\input{header.tex}

\begin{document}

\maketitle






\begin{abstract}
    We consider a 
    suspension of spherical inertialess particles in a Stokes flow on the torus $\T^3$. The particles perturb a linear extensional flow due to their rigidity constraint. 
    Due to the singular nature of this perturbation, no mean-field limit for the behavior of the particle orientation can be valid. This contrasts with widely used models in the literature such as the FENE and Doi models and similar models for active suspensions.
    The proof of this result is based on the study of the mobility problem of a single particle in a non-cubic torus, which we prove to exhibit a nontrivial coupling between the angular velocity and a prescribed strain. 
\end{abstract}


\input{1Setting.tex}
\input{2ProofCorollary}

\input{3ProofTheorem}
\input{4auxiliary}

\section*{Acknowledgements}

The authors thank David G\'erard-Varet for discussions that have led to  the setup of the problem and R.S. thanks him and the IMJ-PRG for the hospitality during the stay in Paris where this article originates.   

All authors are grateful for the great scientific atmosphere they experienced in the Centro de Ciencias de Benasque Pedro Pascual during the 
workshop \enquote{IX Partial differential equations, optimal design and numerics}, where most of this article was conceived.

R.H. is supported  by the German National Academy of Science Leopoldina, grant LPDS 2020-10.

A.M. is supported by the SingFlows project, grant ANR-18-CE40-0027 of the French National Research Agency (ANR).

R.S. is supported by the Deutsche Forschungsgemeinschaft (DFG, German Research Foundation) through the collaborative research
centre ‘The mathematics of emergent effects’ (CRC 1060, Project-ID 211504053).

\appendix

\printbibliography

\includepdf[pages=-]{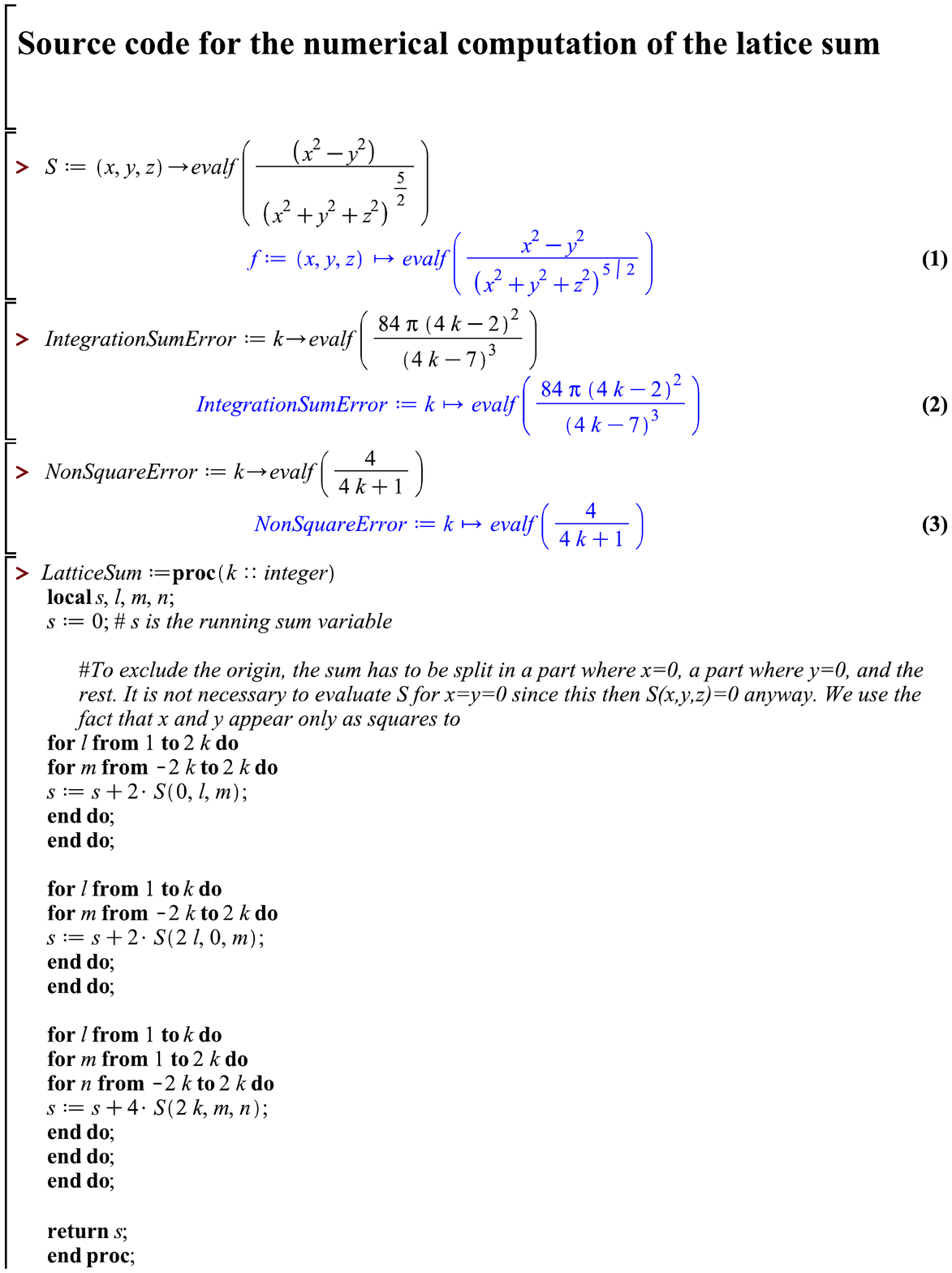}

\end{document}

%% file: header.tex
\usepackage[utf8]{inputenc}
\usepackage[T1]{fontenc}
\usepackage{lmodern}
\usepackage[english]{babel}
\usepackage{csquotes}
\usepackage{amsthm, amssymb, amsmath, amsfonts, mathrsfs, dsfont, esint,textcomp}
\usepackage[colorlinks=true, pdfstartview=FitV, linkcolor=blue, citecolor=blue, urlcolor=blue,pagebackref=false]{hyperref}
\usepackage{mathtools}
\usepackage[shortlabels]{enumitem}
\usepackage[backend=biber,giveninits,maxnames=4,style=alphabetic,isbn=false, date=year, doi=false, eprint= true, url= false]{biblatex}
\usepackage{pdfpages}
 \usepackage{authblk}

\setkomafont{date}{\large}

 \renewcommand{\arraystretch}{1.4}
\usepackage{microtype}

\usepackage[notcite,notref,color]{showkeys}
\definecolor{labelkey}{gray}{.8}
\definecolor{refkey}{gray}{.8}

\definecolor{darkblue}{rgb}{0,0,0.7} 
\definecolor{darkred}{rgb}{0.9,0.1,0.1}
\definecolor{darkgreen}{rgb}{0,0.5,0}

\newtheorem{thm}{Theorem}[section]

\newtheorem{prop}[thm]{Proposition}
\newtheorem{lem}[thm]{Lemma}
\newtheorem{cor}[thm]{Corollary}
\theoremstyle{remark}

\theoremstyle{definition}

\newcommand\norm[1]{\left\lVert#1\right\rVert}
\newcommand\set[1]{\left\{#1\right\}}

\newcommand\abs[1]{\left\lvert#1\right\rvert}

\renewcommand{\le}{\leqslant}

\renewcommand{\leq}{\leqslant}
\renewcommand{\geq}{\geqslant}

\renewcommand{\subset}{\subseteq}

\newcommand{\N}{\mathbb{N}}

\newcommand{\1}{\mathbf{1}}
\newcommand{\R}{\mathbb{R}}

\newcommand{\Z}{\mathbb{Z}}
\newcommand{\T}{\mathbb{T}}
\renewcommand{\S}{\mathbb S}

\renewcommand{\P}{\mathcal{P}}

\renewcommand{\div}{\mathop{\rm div}}

\newcommand{\dd}{\, \mathrm{d}}

\newcommand{\dmin}{d_{\mathrm{min}}}

\newcommand{\De}{\mathrm{De}}

\DeclareMathOperator{\Sym}{Sym}

\DeclareMathOperator{\curl}{curl}

\DeclareMathOperator{\Id}{Id}

\DeclareMathOperator{\dv}{div}

\DeclareMathOperator*{\supp}{supp}

\DeclareMathOperator{\pert}{pert}

\numberwithin{equation}{section}

\mathtoolsset{showonlyrefs}

\bibliography{bib.bib}

\title{Non-existence of mean-field models for particle orientations in suspensions}
\author[1]{Richard M. H\"ofer\thanks{hoefer@imj-prg.fr}}
\author[1]{Amina Mecherbet\thanks{mecherbet@imj-prg.fr}}
\author[2]{Richard Schubert\thanks{schubert@iam.uni-bonn.de}}
\affil[1]{Institut de Mathématiques de Jussieu -- Paris Rive Gauche, Université Paris Cit\'e, France}
\affil[2]{Institute for Applied Mathematics, University of Bonn, Germany}

%% file: 1Setting.tex
\section{Introduction}

It is well known that inertialess rigid particles suspended in a fluid change the rheological properties of the fluid flow.
For passive non-Brownian particles this accounts to an increased viscous stress. In more complex models like active (self-propelled) particles or non-spherical Brownian particles, an additional active or elastic  stress arises that renders the fluid viscoelastic.
Over the last years, considerable effort has been invested into the rigorous derivation of effective models for suspensions. This has been quite successful regarding the derivation of effective fluid equations in models when only a snapshot in time is studied for a prescribed particle configuration or for certain toy models that do not take into account the effects of the fluid on the particle evolution (see e.g. \cite{HainesMazzucato12, NiethammerSchubert19, HillairetWu19, Gerard-VaretHillairet19, Gerard-VaretMecherbet20,
DuerinckxGloria20,
Gerard-VaretHoefer21, DuerinckxGloria21, Girodroux-Lavigne22, HoferLeocataMecherbet22}). 

Much less is known regarding the rigorous derivation of fully coupled models between the fluid and dispersed phase, although a number of such models have been proposed a long time ago and some of them have been studied extensively in the mathematical literature. 
These models typically consist of a transport or Fokker-Planck type equation for the particle density coupled to a fluid equation incorporating the effective rheological properties. 

The rigorous derivation of such models is so far limited to sedimenting spherical particles where the transport-Stokes system has been established in \cite{Hofer18MeanField, Mecherbet18} to leading order in the particle volume fraction. This system reads
\begin{align}
\left\{
\begin{array}{rcl}
    \partial_t \rho + (u + g) \cdot \nabla \rho = 0, \\
    -\Delta u + \nabla p = g \rho, \quad \dv u = 0,
    \end{array} \right.
\end{align}
where $\rho(t,x)$ is the number density of particles, $g \in \R^3$ is the constant gravitational acceleration. Here, the gravity is dominating over the change of the rheological properties of the fluid, which only appears as a correction to the next order in the  particle volume fraction $\phi$. More precisely, as was shown   in \cite{Hofer&Schubert}, a more accurate description is given by the system
\begin{align} \label{eq:transport.Stokes.Einstein} 
\left\{
\begin{array}{rcl}
    \partial_t \rho + (u + g) \cdot \nabla \rho = 0, \\
    -\dv((2 + 5 \phi \rho)D u) + \nabla p = g \rho, \quad \dv u = 0,
\end{array} \right.
\end{align}
where $D u = 1/2(\nabla u + (\nabla u)^T) $ denotes the symmetric gradient.

For non-spherical particles, the increase of
viscous stress depends on the particle orientations (see e.g. \cite{HillairetWu19} for a rigorous result in the stationary case). Moreover, elastic stresses are observed for non-spherical Brownian particles as well as active stresses for self-propelled particles, both depending on the particle orientation.
Therefore, it is necessary to consider models for the evolution of particle densities $f$ that include the particle orientation. In the simplest case of identical axisymmetric particles, the particle orientation can be modeled by a single vector $\xi \in \S^2$. The model corresponding to \eqref{eq:transport.Stokes.Einstein} then reads
\begin{align} \label{eq:transport.Stokes.Einstein.orientation}
\left\{
\begin{array}{rcl}
    \partial_t f + (u + g) \cdot \nabla f + \dv_{\xi} \left( \left(\frac 1 2 \curl u \wedge \xi + B P_{\xi^\perp}   D u \xi\right) f\right) &=& 0, \\
    -\Delta u  + \nabla p  - \dv (\phi M[f] Du)= g \rho, \quad \dv u &=& 0.
\end{array}\right.
\end{align}
Here, $P_{\xi^\perp}$ denotes the orthogonal projection in $\R^3$ to the subspace $\xi^\perp$ and $B$ is the Bretherton number that depends only on the particle shape ($B=0$ for spheres, $B=1$ in the limit of very elongated particles, see e.g. \cite[Section 3.8] {Graham}). Moreover $M[f]$ is a $4$-th order tensor depending on the particle shape and given in terms of moments of $f$.

A widely used model for Brownian suspensions of rod-like (Bretherton number $B=1$) particles at very small particle volume fraction $\phi$ is the so called
Doi model (see e.g. \cite{DoiEdwards88, constantin2005nonlinear,
helzel2006multiscale,
LionsMasmoudi07,
constantin2007regularity, constantin2008global,
ZhangZhang08, 
OttoTzavaras08,
constantin2009holder, constantin2010global, 
BaeTrivisa12, 
BaeTrivisa13, 
HelzelTzavaras17,  
La19}) that reads (in the absence of fluid inertia)
\begin{equation}
\label{Doi} \def\arraystretch{1.8}
\left\{
\begin{array}{l}
\displaystyle \partial_t f + \dv (u f) + \dv_{\xi} (P_{\xi^\perp} \nabla_x u \xi f) 
	= \frac 1 {\De} \Delta_\xi f + {\frac {\lambda_1} {\De}} \dv_x((\Id + \xi \otimes \xi) \nabla_x f), \\
\displaystyle	-  \Delta u + \nabla p - \dv \sigma = h, \qquad \dv u = 0\\
\displaystyle	\sigma = \sigma_v + \sigma_e = \phi M[f] Du +  {\frac {\lambda_2 \phi}{\De}} \int_{\S^2} (3 \xi \otimes \xi - \Id) f \dd \xi.
\end{array}
\right.
\end{equation}
Here, $\De$ is the Deborah number, $\lambda_1, \lambda_2$ are constants that depend on the particle shape and $h$ is some given source term.
Neglecting the effect of the fluid on the particles, the elastic stress $\sigma_e$ in the Doi model has been recently derived in \cite{HoferLeocataMecherbet22}.
There are very similar models for active suspensions (the Doi-Saintillan-Shelley model) and flexible particles, most prominently the FENE model.
Well-posedness and behavior of solutions to such models have been studied for example in
\cite{JLL02,
JLL04,
JLLO06,
LL07,
saintillan2008instabilities,
saintillan2008instabilities2,
LL11, 
chen2013global, Masmoudi13,
Saintillan2018review,  CotiZelatiDietertGerardVaret22, AlbrittonOhm22}.

\medskip

The main purpose of this paper is to draw attention to the limitations of such fully coupled models like \eqref{eq:transport.Stokes.Einstein.orientation} and \eqref{Doi}
regarding the modeling of the particle orientations through the term $\dv_{\xi} (P_{\xi^\perp} \nabla_x u \xi f)$ (respectively $\dv_{\xi} ( (1/2 \curl u \wedge \xi + B P_{\xi^\perp}   D u \xi) f)$ for $B \neq 1$).
This term derives from the change of orientation for the particles according to the gradient of the fluid velocity.
However, at least partially, this fluid velocity is arising as a perturbation flow due to the presence of the particles themselves that cause the viscous and elastic stresses $\sigma_v$ and $\sigma_e$. These perturbations are typically of order $\phi$.
On the microscopic level, the perturbed fluid velocity is very singular. More precisely, to leading order, it behaves like the sum of stresslets. At the $i$-th particle, it is given by
\begin{align}
    u^{\pert}_N(X_i) \approx \sum_{j \neq i} \nabla \Phi(X_i - X_j) : S_j
\end{align}
where the sum runs over all particles $j$ different from $i$, $\Phi$ is the fundamental solution of the Stokes equation  and $S_j$ are moments of stress induced at the $j$-th particle due to the rigidity constraint (and possibly activeness or flexibility).
Consequently, the change of orientation behaves like 
\begin{align}
    \dot \xi_i= P_{\xi_i^\perp} \xi_i \cdot \nabla u^{\pert}_N(X_i) \approx P_{\xi_i^\perp} \xi_i \sum_{j \neq i} \nabla^2 \Phi(X_i - X_j) : S_j
\end{align}
As $\Phi$ is homogeneous of degree $-1$, this behavior is too singular to expect the \enquote{naive} mean-field limit to be true that would lead to the term $\dv_{\xi} (P_{\xi^\perp} \nabla_x u \xi f)$ (respectively $\dv_{\xi} ( (1/2 \curl u \wedge \xi + B P_{\xi^\perp}   D u \xi) f)$ for $B \neq 1$) in the models \eqref{eq:transport.Stokes.Einstein.orientation} and \eqref{Doi}. These models therefore do not seem to describe correctly the behavior of the particle orientations to first order in the particle volume fraction $\phi$.
Instead, it seems necessary to include terms that depend on the $2$-point correlation function for an accurate description up to order $\phi$. 

This is reminiscent of the \emph{second} order correction in $\phi$ of the effective viscous stress $\sigma_v$ (see \cite{Gerard-VaretHillairet19, Gerard-VaretMecherbet20, DuerinckxGloria21}).
For the evolution of the particle orientation, this phenomenon already appears at the first order, since the particle orientations are only sensitive to the \emph{gradient} of the fluid velocity.

\medskip
In this paper, we will make these limitations rigorous for a toy model.
More precisely, we consider  a model example in which we show the non-existence of any mean field model that incorporates the change of orientations to the leading  order in the perturbation field of the fluid. This model example consists of a suspension of spherical particles in a background flow which is a linear extensional flow.

In a bounded domain $\Omega \subset \R^3$, the problem would read
    \begin{align}
    \left\{
\begin{array}{rcll}
        - \Delta u + \nabla p = 0, \quad \dv  u &=& 0 &\qquad \text{in } \Omega \setminus \bigcup_i B_i, \\
        u(x) &=& \phi^{-1} Ax &\qquad \text{on } \partial \Omega, \\
        D u &=&  0 &\qquad \text{in }  \bigcup_i B_i, \\
\displaystyle     \int_{\partial B_i} \sigma[u] n \dd S=   \int_{\partial B_i} (x-X_i) \wedge \sigma[u] n \dd S&=& 0 &\qquad \text{for all } 1 \leq i \leq N. 
     \end{array}\right.
    \end{align}
where $B_i = B_R(X_i)$ denote the spherical particles and $A \in \Sym_0(3)$ is a symmetric tracefree matrix. The rescaling with the volume fraction $\phi = N R^3$ is introduced in order to normalize the perturbation fluid velocity field $u_{per}$ induced by the particles.

For mathematical convenience, we consider the analogous problem on the torus $\T^3$.  
We attach (arbitrary) orientations $\xi_i \in \S^2$ to the spheres and show that no mean-field limit can exist by proving that for periodically arranged particles on $\Z^3$, the particles do not rotate at all, while for particles arranged periodically on $(2\Z) \times \Z^2$, the particles \emph{do} rotate with a fixed rate.

\subsection{Statement of the main results}

We will work on the toroidal domains
\begin{align}
    \T = (\R/\Z)^3, \qquad    \T_{L} = (\R/2L\Z)^3, \qquad 
    \bar \T_{L} = (\R/4L\Z) \times (\R/2L\Z)^2.
\end{align}
Furthermore we set
\begin{align}
    B = B_1(0), \qquad B_R= B_R(0),
\end{align}
and for definiteness
\begin{align}
    A = \begin{pmatrix}
    0&1&0\\1&0&0\\0&0&0
    \end{pmatrix}.
\end{align}

\begin{thm} \label{th:main}
    For $0 < R < 1/2$ and $\Omega = \T_1$ or  $\Omega = \bar\T_{1}$, let $u \in H^1(\Omega)$ be the unique weak solution to the problem
    \begin{align}\label{eq:uA}
    \left\{
\begin{array}{rcll}
        - \Delta u + \nabla p = 0, \quad \dv  u &=& 0 &\qquad \text{in } \Omega \setminus B_R, \\
        D u &=& A &\qquad \text{in }  B_R, \\
\displaystyle        \int_{\Omega} u \dd x &=&0,& \\
\displaystyle      \int_{\partial B_R} \sigma[u] n \dd S=   \int_{\partial B_R} x \wedge \sigma[u] n\dd S &=& 0.&
      \end{array} \right.
    \end{align}
    \begin{enumerate}[(i)]
        \item \label{it:symmetric.torus}
    If $\Omega = \T_1$, then $\curl u (0) = 0$.
    \item  \label{it:nonsymmetric.torus}
    If $\Omega = \bar\T_1$. 
    then there exists $\bar c > 0$ such that
    \begin{align} \label{angular.velocity}
        \left|\curl u(0) - \frac{R^3}{16} \bar c e_3\right| \leq C R^4 
    \end{align}
    for a constant $C$ independent of $R$.
    \end{enumerate}
\end{thm}
Observe that the factor $16$ in \eqref{angular.velocity} corresponds to the volume of $\bar \T_1$. 

As a consequence of Theorem \ref{th:main} we show the negative result stated in Corollary \ref{cor:main} below, that  we outlined in the introduction. We will use the following notation.
For $N\in \N$, which we think of as the number of particles in a unit cell, we denote by 
\begin{align}
\phi = NR^3
\end{align}
the volume fraction of the particles. Both $R$ and $\phi$ may implicitly depend on $N$. We will make this dependence explicit by a subscript $N$ wherever we feel it is necessary for clarity.
For $R > 0$ and $N \in \N$, let $X_i \in \T$, $1 \leq i \leq N$ be such that
\begin{align} \label{dmin}
    \dmin = \min_{i \neq j} |X_i - X_j| > 2 R.
\end{align}
This ensures that the balls
\begin{align}
     B_i = B_R(X_i)
\end{align}
do not intersect nor touch each other.
Moreover, for $1 \leq i \leq N$, let $\xi_i^0 \in \S^2$. The associated initial empirical density $f_N^0 \in \P(\R^3 \times \S^2)$ is given by
\begin{align}
    f_N^0 = \frac 1 N \sum_i \delta_{X_i} \otimes \delta_{\xi_i^0}.
\end{align}
Consider the dynamics
\begin{align} \label{xi}
    \frac {\dd}{\dd t} \xi_i =  \omega_i\wedge \xi_i 
\end{align}
where $\omega_i = \tfrac 12\curl u(X_i)$ for the solution $u \in \dot H^1(\T)$ to the problem
    \begin{align} \label{u:micro}
    \left\{
\begin{array}{rcll}
        - \Delta u + \nabla p = 0, \quad \dv  u &=& 0 &\qquad  \text{in } \T \setminus \bigcup_i B_i, \\
        D u &=& \phi^{-1} A &\qquad \text{in }  \bigcup_i B_i, \\
 \displaystyle           \int_{\T} u \dd x&=&0,& \\
\displaystyle      \int_{\partial B_i} \sigma[u] n \dd S =   \int_{\partial B_i} (x-X_i) \wedge \sigma[u] n\dd S &=& 0 &\qquad \text{for all } 1 \leq i \leq N. 
    \end{array} \right.
    \end{align}
Notice that $u$ is the normalized perturbation field induced by the particles with respect to the background flow $-Ax$. We then write 
\begin{align}\label{eq:def_f_N}
    f_N(t) := \frac 1 N \sum_i \delta_{X_i} \otimes \delta_{\xi_i(t)}.
\end{align}

In the following, we denote by $W_p$, $p \in [1,\infty]$ the usual $p$-Wasserstein distance (cf. for example \cite{Santambrogio15}) on the space of probability measures $f \in \P(\T\times \S^2)$. We recall that  $W_p \leq W_q$ for $p \leq q$ and that $W_1$ metrizes weak convergence of measures.
\begin{cor} \label{cor:main}
    For all sequences $R_N \to 0$ with $\phi_N \to 0$, there exist constants $c,T>0$ and $(X_1,\cdots,X_N) \in \T^{N}$ and $(\xi_1^0,\cdots,\xi_N^0) \in (\S^2)^N$, $N\in \N$, 
    such that the associated empirical measures $f_N\in C([0,\infty);\P(\T \times \S^2))$ defined by \eqref{eq:def_f_N} and \eqref{xi} satisfy the following properties.
    \begin{enumerate}[(i)]
        \item \label{it:min.dist} With $\dmin$ defined as in \eqref{dmin}
    \begin{align}
        \dmin \geq c N^{-1/3}.
    \end{align} 
    \item  \label{it:initial.convergence} There exists $f_0 \in \P(\T \times \S^2) \cap C^\infty(\T \times \S^2)$ such that $W_\infty(f_N^0,f^0) \to 0$.
    \item \label{it:accumulation.points} There exist at least two distinct accumulation points of $f_N$. More precisely, there exist subsequences $f_{N_k}$, 
    $f_{{\bar N}_k}$ and $f, \bar f \in C^\infty([0,\infty);\T \times \S^2)$ which satisfy
    \begin{align}
    \sup_{t \in [0,T]} W_\infty(f_{N_k}(t),f(t)) + W_\infty(f_{\bar N_k}(t),\bar f(t)) \to 0,  \label{eq:accumulation.points}\\
         W_1(f(t),\bar f(t)) \geq c t \quad \text{for all } t \leq T.  \label{W_1(f,bar.f)}
    \end{align}
    \end{enumerate}
\end{cor}
Several remarks are in order.
\begin{itemize}
\item  Observe that the corollary indeed shows that no general mean-field model can describe the effective behavior of the microscopic system \eqref{xi}--\eqref{u:micro} since there is a sequence $f_N$ that on the one hand converges at the initial time to some $f_0$ but that on the other hand has at least two distinct accumulation points for $0 < t \leq T$.
In particular, the \enquote{naive} mean-field limit
\begin{align} \label{eq:momentum} \def\arraystretch{1.8}
    \left\{
\begin{array}{rcl}
\displaystyle \partial_t f + \dv_{\xi} \left(\frac 1 2\curl u \wedge \xi f\right) 
	&=& 0, \\
\displaystyle	-  \Delta u + \nabla p  =  - 5 \dv \left(A  \int_{S^2} f \dd \xi \right) , \quad  \dv u &=& 0
	\end{array} \right.
\end{align}
cannot hold true. Here, the factor $5$ arises as the relation between the strain and stress of an isolated sphere in an infinite fluid, cf. \eqref{eq:transport.Stokes.Einstein}.
Note that the momentum equation in \eqref{eq:momentum} can be obtained from $-\dv((2 + 5 \phi \rho )Dv) + \nabla p   $ from the ansatz $v(x) = - Ax + \phi u(x)$ upon taking $\phi \to 0$.

\item  Condition \ref{it:min.dist} ensures that the non-convergence is not caused by particle clusters but appears for well-separated particles. 

\item An adapted version of the statement remains true  when one takes into account the time-evolution of the particle positions according to 
\begin{align}
    \frac{\dd}{\dd t} X_i = u(X_i).
\end{align}
Indeed, as pointed out above, regarding the translations, the \enquote{naive} mean-field limit does hold, at least as long as the particles remain well-separated (cf. \cite{Hofer&Schubert}). In the proof of the corollary, we only consider distributions of particles which are periodic in space. Since periodicity is preserved under the dynamics, such clustering cannot occur.
Nevertheless, we  restrict ourselves to the case of fixed particle centers for the sake of the simplicity of the presentation.

\item The regularity of the limiting density $f$ \emph{strengthens} the statement. Indeed the simplest approach would be to consider particles that all have the same orientation, which leads to a delta distribution in orientation for $f$. 
\end{itemize}

%% file: 2ProofCorollary.tex
\section{Proof of Corollary \ref{cor:main}}

Let $R_N \to 0$ with $\phi_N \to 0$ be given.

Let $f_0 = h(\xi)$ for some $h(\xi) \in \P(\S^2) \cap C^\infty(\S^2)$ that will be chosen later.
For $k \in \N$ let $N_k = (k)^3$ and $\bar N_k = 4 k^3$ and let
$\{X_i\}_{i =1}^{N_k} = (\Z/k)^3\subset \T$
and $\{\bar X_i\}_{i =1}^{\bar N_k} = (\Z/k) \times (\Z/2k)^2\subset \T$.  Define
\begin{align}
    f_{N_k}^0 = \frac{1}{N_k} \sum_i \delta_{X_i} \otimes \delta_{\xi_i^0}, \\
    f_{\bar N_k}^0 = \frac{1}{\bar{N}_k} \sum_i \delta_{\bar X_i} \otimes \delta_{\bar \xi_i^0}
\end{align}
where the initial orientations $\xi_i^0$ and $\bar \xi_i^0$ are chosen in such a way that both $W_\infty(f_{N_k}^0,f^0) \to 0$ and  $W_\infty(f_{\bar N_k}^0,f^0) \to 0$ (for example by taking samples  of initial  distributions $\xi_i^0$ and $\bar \xi_i^0$ which are i.i.d. distributed with law $h$).

By a suitable choice of $f_N^0$ for $N \not \in \{k^3 : k \in \N\} \cup \{4 k^3 : k \in \N\}$, we can ensure that items \ref{it:min.dist} and \ref{it:initial.convergence} are satisfied.
We will show that  item \ref{it:accumulation.points} holds true with $f(t,\cdot) := 1 \otimes h$ and $\bar f(t) :=1 \otimes (h\circ e^{-\frac{\bar c} 2 M t})$,
where $\bar c$ is the constant from Theorem \ref{th:main} and
\begin{align}
    M = \begin{pmatrix}
    0&-1&0\\1&0&0\\0&0&0
    \end{pmatrix}.
\end{align}
is the unique skew-symmetric matrix satisfying $M v = e_3 \wedge v$ for all $v \in \R^3$. 

In order to show \eqref{eq:accumulation.points}, it suffices to prove that
\begin{align}
    \dot \xi_i = 0 \qquad \text{ for all } 1 \leq i \leq N_k, \\
    \dot {\bar \xi}_i = \frac {\bar c} 2  e_3 \wedge \xi_i + O(R_N) \qquad \text{ for all } 1 \leq i \leq \bar N_k.
\end{align}
This is an immediate consequence of Theorem \ref{th:main}.

It remains to prove \eqref{W_1(f,bar.f)}. We use the well-known characterization (cf. \cite[Equation (3.1)]{Santambrogio15})
\begin{align}
    W_1(f(t),\bar f(t)) = \sup\left\{ \int_{\T \times \S^2} (f(t) - \bar f(t)) \varphi \dd \xi \dd x : \varphi \colon \T \times \S^2 \to \R \text{ is } 1-Lipschitz \right\}.
\end{align}
Choosing $\varphi(x,\xi) = \xi_1$ yields
\begin{align}
    W_1(f(t),\bar f(t)) \geq \int_{\S^2} \left(\xi_1 - \left(e^{ \frac {\bar c} 2M t} \xi\right)_1 \right) h(\xi) \dd \xi =: g(t) 
\end{align}
We observe that
\begin{align}
    g'(0) =  \int_{\S^2} \frac {\bar c} 2 \xi_2 h(\xi) \dd \xi.
\end{align}
Since $h \in \P(S^2) \cap C^\infty(S^2)$ was arbitrary, we may choose $h$ in such a way that $g'(0) > 0$ (e.g. by taking $h$ with  $\supp h \subset \{\xi_2 > 0\}$). Then \eqref{W_1(f,bar.f)} holds.



%% file: 3ProofTheorem.tex
\section{Proof of Theorem \ref{th:main}}
\subsection{Proof of the first item}\label{subsec:item1}
For this subsection, denote by $u_A$ the solution to \eqref{eq:uA} for $\Omega=\T_1$, i.e.
 \begin{align*}
     \left\{
\begin{array}{rcll}
        - \Delta u_A + \nabla p = 0, \quad \dv  u_A &=& 0 &\qquad \text{in }  \T_1 \setminus B_R, \\
        D u_A &=& A &\qquad \text{in }  B_R, \\
        \int_{ \T_1} u_A &=&0,& \\
      \int_{\partial B_R} \sigma[u_A] n \dd S=   \int_{\partial B_R} x \wedge \sigma[u_A] n \dd S&=& 0.&
      \end{array} \right.
    \end{align*}
Let $S\in SO(3)$ be any rotation matrix that leaves the torus $\T_1$ invariant. Then, we have 
$$
u_{S^\top A S}(x)=S^\top u_A(Sx),\qquad  u_{-A}(x)=-u_A(x).
$$
Therefore, denoting by $\omega[u_A]:=\frac 12\curl u_A(0)$, the angular velocity of the particle associated with $u_A$, one can show that
\begin{equation}\label{eq:relation_angular_velocity_R}
\omega[u_{S^\top A S}]=S^\top \omega[u_A],\qquad  \omega[u_{-A}]=-\omega[u_A].
\end{equation}
Taking $S= S_k$, $k=1,2,3$,  
\begin{align}  \label{R}
 S_1= \begin{pmatrix}1& 0&0 \\0&-1&0\\0&0&-1 \end{pmatrix}, \qquad  S_2 = \begin{pmatrix}-1& 0&0 \\0&1&0\\0&0&-1 \end{pmatrix}, \qquad  S_3 = \begin{pmatrix}0& -1&0 \\1&0&0\\0&0&1 \end{pmatrix}.
\end{align}
which all have the property $S_k^T A S_k = - A$, we deduce from \eqref{eq:relation_angular_velocity_R}  that the three components of $\omega[A]$ are vanishing.

\subsection{Proof of the second item}

Let now $u_A$ be the solution to \eqref{eq:uA} with $\Omega=\bar\T_1$. The fact that the first two components of $\omega[u_A]:=\tfrac 12 \curl u_A(0)$ vanish can be shown by the same argument as in Subsection~\ref{subsec:item1}, considering  $S_1$ and $S_2$ in \eqref{R} that leave $\bar\T_1$ invariant.

For convenience, we consider the rescaled torus $ \bar\T_{1/R}$ instead of $ \bar\T_{1}$ and set $L=1/R$. More precisely, we consider $u$ to be the solution to 
 \begin{align}\label{eq:u}
     \left\{
\begin{array}{rcll}
        - \Delta u + \nabla p = 0 , \quad    \dv  u &=& 0 &\qquad \text{in } \bar\T_{L} \setminus B, \\
        D u &=& A &\qquad \text{in }  B, \\
  \displaystyle      \int_{\bar\T_{L}} u \dd x&=&0, \\
  \displaystyle       \int_{\partial B} \sigma[u] n \dd S=   \int_{\partial B} x \wedge \sigma[u] n \dd S&=& 0.
\end{array}\right.
\end{align}
By rescaling it remains to prove the following claim:
\begin{align}
        |\curl u(0) + \frac{1}{L^3}\bar c e_3| \leq C L^{-4}. 
\end{align}
The proof is based on a good explicit approximation of $u$. Here it is useful to think of $B$ not as a \emph{single} particle in the torus but as one of infinitely many periodically distributed particles in $\R^3$. Consequently we will in the following consider functions that a priori are defined on $\R^3$ even if they \emph{turn out} to be periodic and can thus be considered as functions defined on $\bar \T_L$. If the volume fraction of the particles is small, the flow field $u$ is well approximated 
by the superposition of the single particle solutions $w$ of the problem

\begin{align}\label{eq:w}
     \left\{
\begin{array}{rcll}
        - \Delta w + \nabla p = 0, \quad  \dv  w &=& 0  &\qquad \text{in } \R^3 \setminus B,\\
        D w &=& A &\qquad \text{in }  B, \\
  \displaystyle         \int_{\partial B} \sigma[w] n \dd S=   \int_{\partial B} x \wedge \sigma[w] n \dd S&=& 0.
\end{array}\right.
\end{align}
We emphasize that $w(x)=Ax$ in $B$ and we have (see e.g. \cite[Eq. (1.11)]{NiethammerSchubert19})
\begin{equation}\label{eq:approximation_w}
w(x)=-\frac{20\pi}{3}  \nabla \Phi(x):A +  R[A](x)
\end{equation}
where $R[A](x)$ is homogeneous of degree $-4$ and $\Phi$ is the fundamental solution of the Stokes equations, i.e.
\begin{align}
    \Phi(x) &= \frac 1 {8 \pi} \left( \frac{\Id}{|x|} + \frac{x \otimes x}{|x|^3} \right), \\
     (\nabla \Phi(x):A)_i := \partial_k \Phi_{ji}(x) A_{jk} &= - \frac 3 {8 \pi} \frac{x_i x_j x_k A_{jk}}{|x|^5}.\label{eq:gradphi}
\end{align}
We set 
\begin{align}\label{eq:Lambda}
\Lambda_L= \{(y_1,y_2,y_3),\, y_1 \in 4 L\Z, y_2 \in 2 L \Z, y_3 \in 2 L \Z \}
\end{align}
and define the superposition of single particle solutions $\tilde u $ for $x\in  \R^3$ by
\begin{align}\label{eq:ubar}
\begin{aligned}
\bar{u}(x)=\underset{y \in \Lambda_L}{\sum}& \left( w(x-y)+\frac{20\pi}{3} \fint_{Q_y} \nabla \Phi(x-z):A \dd z\right.\\
&\left. - \fint_{Q_y} \left\{w(x'-y)+\frac{20\pi}{3} \fint_{Q_y} \nabla \Phi(x'-z):A \dd z \right\} \dd x'  \right) 
\end{aligned}
\end{align}
with $Q_y= y + [-2L,2L] \times [-L,L]^2$. Here we subtract iterated means (the sum of which formally vanishes) in order to make the sum absolutely convergent. 

The approximation $\bar u$ is convenient because on the one hand we can profit from the fact that its building blocks satisfy the PDE \eqref{eq:w} and hence $\bar u$ itself also satisfies a PDE (see Proposition~\ref{prop_bar_u} below) in order to compare it to $u$ (see Proposition~\ref{pro:u.bar.u} below).  On the other hand $w$ and hence $\bar u$ are explicit and, by Lemma~\ref{le:u.tilde.u.bar} below, it is close (inside of $Q_0$) to the following even simpler explicit approximation (shifted by $w$)
\begin{align} \label{def:u.tilde}
\begin{aligned}
\tilde {u}(x)=&\frac{20\pi}{3} \fint_{Q_0} \nabla \Phi(x-z):A \dd z\\
&-\frac{20\pi}{3}  \underset{y \in \Lambda_L \setminus\{0\}}{\sum} \left(\nabla \Phi(x-y)-  \fint_{Q_y} \nabla \Phi(x-z)\dd z \right.\\
&\qquad\qquad\qquad\left.- \fint_{Q_y} \left\{\nabla \Phi(x'-y) -  \fint_{Q_y} \nabla \Phi(x'-z) \dd z \right\} \dd x'  \right):A,
\end{aligned}
\end{align}
which makes explicit lattice computations accessible.
\begin{prop}\label{prop_bar_u}
The functions $\bar u$ from \eqref{eq:ubar} and $\tilde u$ from \eqref{def:u.tilde} are well-defined and satisfy $\bar u, \tilde u \in W^{1,\infty}(Q_0)$.  Moreover, $\bar u$ is periodic and is a weak solution to
 \begin{align} \label{eq:baru}
      \left\{
\begin{array}{rcl}
        - \Delta \bar u + \nabla = 0, \quad \dv  \bar{u} &=& 0  \qquad \text{in } \bar\T_{L} \setminus B,\\
    \displaystyle       \int_{\bar\T_{L}} \bar u \dd x&=&0, \\
  \displaystyle       \int_{\partial B} \sigma[\bar u] n\dd S =   \int_{\partial B} x \wedge \sigma[\bar u] n \dd S&=& 0.
    \end{array}\right.
    \end{align}
\end{prop}
\begin{proof} We consider only $\bar u$, the argument for $\tilde u$ is analogous.

First observe that each term in the series is well defined. Let $x\in Q_0$. Then, the only term that needs further justification is the term corresponding to $y=0$.  in this case $w(x)+\frac{20\pi}{3} \fint_{Q_0} \nabla \Phi(x-z):A \dd z$ is well defined and uniformly bounded with respect to  $x \in Q_0$ since $ \nabla \Phi(x-z)$ is locally integrable. To see that the derivative is well defined as well let us consider the union of the $27$ neighbouring cells around $Q_0$,
\begin{align}
    \tilde Q_0 := \{ z \in \R^3 : z \in Q_y \text{ for some } y \in \Lambda_L \text{ with } { Q_y }\cap { Q_0} \neq \emptyset \} = [-4L,4L] \times [-2L,2L]^2.
\end{align}
Then, via integration by parts
$$\int_{\tilde Q_0} \nabla^2 \Phi(x-z):A \dd z = -\int_{\partial \tilde Q_0} \nabla\Phi(x-z) (A n) \dd S_z$$ which is uniformly bounded for $x \in Q_0$.

We now show that the series is absolutely convergent. Regarding the gradient of each term of the series \eqref{eq:ubar}, we observe that by \eqref{eq:approximation_w}, for $y \neq 0$  it holds
\begin{align}
\left| \nabla w(x-y)+\frac{20\pi}{3} \fint_{Q_y} \nabla^2 \Phi(x-z):A \dd z   \right|
 &\leq C|A| \frac{L}{|x-y|^4} 
\end{align}
This shows that the sum of the gradients is absolutely convergent. Because each term in the sum \eqref{eq:ubar} has vanishing average over the respective cell $Q_y$, the terms of the series decay with the same rate $\abs{x-y}^{-4}$, and thus also the sum itself is absolutely convergent. 

The periodicity is immediate from the construction. To see that $\bar u$ has vanishing mean, it is enough to notice that the $Q_0$ term has vanishing mean over $Q_0$ and that both $w$ and $\nabla \phi$ are skewsymmetric in the sense that they change sign under the transformation $x\mapsto -x$. In order to show that $\bar u$ satisfies the other identities in \eqref{eq:baru}, we note that both $w$ and $\nabla \Phi$ satisfy Stokes equation outside $B$, and we emphasize that the term $x \mapsto \int_{\tilde Q_y} \nabla \Phi(x-z):A \dd z$ satisfies $-\Delta v+ \nabla q= \div (1_{\tilde Q_y} A)$. By summation, no source term is induced in ${Q_0}$. 
\end{proof}

\begin{lem} \label{le:u.tilde.u.bar} For $\bar u$ defined in \eqref{eq:ubar} and $\tilde u$ defined in \eqref{def:u.tilde}, the following estimates hold.  
    \begin{align}
        \|\nabla (\bar u - \tilde u) - A\|_{L^\infty(B)} \lesssim L^{-5}, \label{est:u.bar.u.tilde}\\
        \|\nabla  \tilde u - \nabla \tilde u(0)\|_{L^\infty(B)} \lesssim L^{-4} \label{est:u.tilde.origin}, \\
        \|\nabla  \tilde u \|_{L^\infty(B)} \lesssim L^{-3} \label{est:nabla.u.tilde}.
    \end{align}
\end{lem}
\begin{proof}
Estimate~\eqref{est:u.bar.u.tilde} follows immediately from the definitions of $\bar u$ and $\tilde u$ as well as the fact that $R[A]$ in \eqref{eq:approximation_w} is homogeneous of degree $-4$. For estimate \eqref{est:u.tilde.origin} we split as follows for $x\in B$
\begin{align}\label{estimate_bar_u-tilde_u}
\begin{aligned}
\nabla \tilde u(x) - \nabla \tilde{u} (0) &= \frac{20\pi}{3} \left (\fint_{Q_0} (\nabla^2 \Phi(x-z)-\nabla^2 \Phi(-z)):A \dd z \right)\\
&\quad+\frac{20\pi}{3}\underset{y \in \Lambda_L \setminus \{0\}}{\sum}  \biggl ((\nabla^2 \Phi(x-y)-\nabla^2 \Phi(-y))  \\
& \qquad \qquad \qquad \qquad - \fint_{Q_y} (\nabla^2 \Phi(x-z)-\nabla^2 \Phi(-z)) \dd z \biggr):A.
\end{aligned}
\end{align}
We deal with the first term by first applying an integration by parts in order to get 
$$
\frac{1}{L^3} \left|\int_{\partial Q_0} (\nabla \Phi(x-z)-\nabla \Phi(-z)) \cdot (An) \dd z \right| \leq C \frac{|A|}{L^3} \int_{\partial Q_0} \left(\frac{1}{|x-z|^3} + \frac{1}{|z|^3} \right) |x| \dd S \leq C \frac{|A|}{L^4}.
$$
The remainder in \eqref{estimate_bar_u-tilde_u} can be handled directly by the same estimates as in the proof of Proposition~\ref{prop_bar_u}. Estimate \eqref{est:nabla.u.tilde} is shown analogously.
\end{proof}

\begin{prop} \label{pro:u.bar.u}
For $u$ satisfying \eqref{eq:u} and $\bar u$ defined in \eqref{eq:ubar}, it holds that
    \begin{align}
        \|\nabla (u - \bar u)\|_{H^1(\bar\T_{L})} \leq C \|A -  D \bar u\|_{L^2(B)}.
    \end{align}
\end{prop}
\begin{proof}
Let $v := \bar u - u$. Then by definition of $u$ and Proposition \ref{prop_bar_u}, $v$ satisfies
         \begin{align}
         \left\{\begin{array}{rcl}
        - \Delta v + \nabla p = 0, \quad  \dv  v &=& 0 \qquad \text{in } \bar\T_{L} \setminus B, \\
    \displaystyle  \int_{\partial B} \sigma[v] n \dd S=   \int_{\partial B} x \wedge \sigma[\bar v] n \dd S&=& 0.
    \end{array} \right.
    \end{align}
    By standard considerations, $\|\nabla v\|_{L^2(\bar\T_{L})} \leq \|\nabla w\|_{L^2(\bar\T_{L})}$ for all $w \in H^1(\bar\T_{L})$ with $D v = D w$ in $B$ and such a function exists with 
    \begin{align}
        \|\nabla w \|_{L^2(\bar\T_{L})} \lesssim 
        \| D v\|_{L^2(B)} = \|A -  D \bar u\|_{L^2(B)}. 
    \end{align}
    We refer to \cite[Lemma 4.6]{NiethammerSchubert19} for details.
\end{proof}

\begin{proof}[Proof of Theorem \ref{th:main}\ref{it:nonsymmetric.torus}]

Direct computation starting with \eqref{eq:gradphi} yields $$\curl (\nabla \Phi(z):A)= -\frac{3}{4\pi} \frac{Ax \wedge x}{|x|^5}.$$
Recalling that $ A=\begin{pmatrix} 0&1&0\\1&0&0\\0&0&0\end{pmatrix}$ we get 
\begin{align}
        \frac 15\left(\curl \tilde u(0)\right)_3 = \fint_{Q_0} \frac{z_1^2-z_2^2}{\abs{z}^5}\dd z-\sum_{y\in \Lambda_L\setminus \{0\}} \left(\frac{y_1^2-y_2^2}{\abs{y}^5} - \int_{Q_y} \frac{z_1^2-z_2^2}{\abs{z}^5} \dd z \right).
\end{align}
Thus, setting $\bar c = 5 c_0$ with $c_0$ being the constant from  Lemma~\ref{le:sign} below, we find, using that $\curl u$ is constant in $B$,
\begin{align}
        \Bigl|\curl u(0) - \frac{\bar  c}{L^3} e_3\Bigr| \leq \left| \fint_{B} (\curl u - \curl \bar u)  \dd x \right|+ \left|\fint_{B} (\curl \bar u - \curl \tilde u) \dd x \right|+ \left|\fint_{B} (\curl \tilde u - \curl \tilde u(0))\dd x \right|. 
\end{align}
Using Lemma~\ref{le:curl.B_R.B_1} below, as well as Proposition~\ref{pro:u.bar.u} and Lemma~\ref{le:u.tilde.u.bar}, we have
\begin{align}
       \left| \fint_B (\curl u - \curl \bar u \dd x \right|&= \left| \fint_{B_L} (\curl u - \curl \bar u) \dd x \right| \\
        &\lesssim L^{-3/2} \|\nabla(u - \bar u)\|_{L^2(\bar \T_{L})} \lesssim 
        L^{-3/2} \|A - D\bar u\|_{L^2(B)} \\
        &\lesssim L^{-3/2} \left(\|\nabla (\tilde u -\bar u) - A\|_{L^2(B)} +  \|\nabla \tilde u \|_{L^2(B)} \right) \lesssim L^{-9/2}.
\end{align}
    Combining this with the estimates for $\fint_{B} |\curl \bar u - \curl \tilde u)|\dd x$ and  $\fint_{B} |\curl \tilde u - \curl \tilde u(0)|\dd x$ provided by Lemma \ref{le:u.tilde.u.bar}, we conclude 
    \begin{align}
         \Bigl|\curl u(0) - \frac{\bar  c}{L^3} e_3\Bigr| \lesssim L^{-4}.
    \end{align}
    This implies the assertion by rescaling to the torus $\bar\T_{1}$ and the ball $B_R$.
\end{proof}

%% file: 4auxiliary.tex
\section{Auxiliary results}

\begin{lem} \label{le:curl.B_R.B_1}
Let $w\in H^1(B_R)$ with $R>1$ and satisfying
\begin{align*}
      \left\{
\begin{array}{rcl}
    -\Delta w+\nabla p=0\quad,     \dv u&=&0\qquad \text{in }B_R\setminus B,\\
  \displaystyle  \int_{\partial B} x\wedge \sigma[w]n\dd S&=&0.
    \end{array} \right.
\end{align*}
Then 
\begin{align*}
    \fint_{B} \curl w\dd x=\fint_{B_R} \curl w\dd x.
\end{align*}
\end{lem}
\begin{proof}
Let $\omega\in \R^3$ and $\varphi\in H^1_0(B_R)$ the solution to 
\begin{align*}
      \left\{
\begin{array}{rcll}
    -\Delta \varphi+\nabla p=0\quad, \    \dv \varphi&=&0&\qquad \text{in }B_R\setminus B,\\
    \varphi&=&\omega\wedge x &\qquad \text{in }B.
        \end{array} \right.
\end{align*}
It is easy to verify that the solution $\varphi$ in $B_{R} \setminus B$ is given by 
\begin{align*}
    \varphi(x)=\frac{R^3}{R^3-1}\left(\frac{1}{\abs{x}^3}-\frac{1}{R^3}\right)\omega\wedge x.
\end{align*}
The corresponding normal stress on $B$ and $B_R$ is $\sigma[\varphi]n=-3\frac{R^3}{R^3-1}\omega\wedge n$ and $\sigma[\varphi]n=-3\frac{1}{R^3-1}\omega\wedge n$, where $n$ is the outward unit normal to $B$ and $B_R$, respectively. 
We compute
\begin{align*}
\omega\cdot\int_B \curl w\dd x&=\int_{\partial B}\omega\cdot(n\wedge w)\dd S=\int_{\partial B}w\cdot(\omega\wedge n)\dd S=-\frac{R^3-1}{3R^3}\int_{\partial B}w\cdot(\sigma[\varphi]n)\dd S\\
&=\frac{R^3-1}{3R^3}\int_{B_R\setminus B}\nabla w\colon \nabla \varphi\dd x- \frac{R^3-1}{3R^3}\int_{\partial B_R}  w\cdot(\sigma[\varphi]n)\dd S\\
&=\frac{R^3-1}{3R^3}\int_{\partial B}(\sigma[w]n)\cdot \varphi+ \frac{1}{R^3}\int_{\partial B_R}  w\cdot(\omega\wedge n)\dd S\\
&=\frac{R^3-1}{3R^3}\int_{\partial B}(\sigma[w]n)\cdot (\omega\wedge n)+ \frac{1}{R^3}\omega\cdot\int_{\partial B_R}  n\wedge w\dd S\\
&=\frac{R^3-1}{3R^3}\omega\cdot\int_{\partial B}n\wedge(\sigma[w]n)\dd S+ \frac{1}{R^3}\omega\cdot\int_{B_R}  \curl w\dd S.
\end{align*}
Since the first term vanishes and $\omega$ was arbitrary, this proves the statement.
\end{proof}

\begin{lem} \label{le:sign}
Let $\Lambda_L$ be the lattice defined in \eqref{eq:Lambda}. There exists a constant $c_0>0$ such that
\begin{align*}
    \fint_{Q_0} \frac{z_1^2-z_2^2}{\abs{z}^5}\dd z-\sum_{y\in \Lambda_L\setminus \{0\}} \left(\frac{y_1^2-y_2^2}{\abs{y}^5}-\fint_{Q_y}\frac{z_1^2-z_2^2}{\abs{z}^5}\dd z\right)=\frac{1}{L^3}c_0.
\end{align*}
\end{lem}
\begin{proof}
Since the individual terms in the sum decay like $\abs{y}^{-4}$ the sum exists and converges absolutely. 

By homogeneity, it is enough to consider $L=\frac 12$. More precisely, we denote the rescaled lattice $\Lambda=\{(y_1,y_2,y_3)\colon y_1\in 2\Z,\;y_2,y_3\in \Z\}$, the rescaled cells $Q'_y=y+[-1,1]\times[-\tfrac 12,\tfrac 12]^2$, and for $y\in \Lambda$
\begin{align*}
     S(y)&=S(y_1,y_2,y_3)=\frac{y_1^2-y_2^2}{\abs{y}^5},\qquad
     S'(y)=S(y)-\fint_{Q'_y}S(z)\dd z.
\end{align*} 
Then, it is enough to show that 
\begin{align}\label{eq:L}
    c_0'\coloneqq -\fint_{Q'_0}S(z)\dd z+\sum_{y\in \Lambda\setminus \{0\}} S'(y)<0.
\end{align}
The strategy to prove \eqref{eq:L} is to show that is is enough to sum over all lattice points in a large enough cube and to estimate the remaining sum. 
We start by proving that, if summing over a finite cube, it is possible to ignore the mean intergals in the sum up to a small error.
We denote $\abs{y}_\infty=\max\{\abs{y_1},\abs{y_2},\abs{y_3}\}$ and rewrite
\begin{align}\label{eq:sum_seq}
    c_0'=\lim_{k\to \infty} -\fint_{Q'_0}S(z)\dd z+\sum_{y\in \Lambda\setminus \{0\},\abs{y}_\infty\le 2k} S'(y).
\end{align}
The contribution of the mean integral terms at level $k$ is 
\begin{align}\label{eq:int_contr}
     \fint_{Q'_0}S(z)\dd z+&\sum_{y\in \Lambda\setminus \{0\},\abs{y}_\infty\le 2k} \fint_{Q'_y}S(z)\dd z\\&=\frac 12\int_{[-2k-1,2k+1]\times [-2k-\frac 12,2k+\frac 12]^2} S(z)\dd z\\
     &=\frac 12\int_{[-2k-\frac 12,2k+\frac 12]^3} S(z)\dd z+\int_{[2k+\frac 12,2k+1]\times[-2k-\frac 12,2k+\frac 12]^2} S(z)\dd z\\
     &=\int_{[2k+\frac 12,2k+1]\times[-2k-\frac 12,2k+\frac 12]^2} S(z)\dd z.
\end{align}
The last identity holds since $S(y_1,y_2,y_3)=-S(y_2,y_1,y_3)$ and hence the integral vanishes on every domain that is invariant under the exchange of $y_1,y_2$.
Since 
\begin{align}\label{eq:int_remainder}
     \abs{\int_{[2k+\frac 12,2k+1]\times[-2k-\frac 12,2k+\frac 12]^2} S(z)\dd z}&\le \int_{[2k+\frac 12,2k+1]\times[-2k-\frac 12,2k+\frac 12]^2}\frac{1}{\abs{z}^3}\dd z\\
     &\le  \frac{1}{(2k+\frac 12)^3}\int_{[2k+\frac 12,2k+1]\times[-2k-\frac 12,2k+\frac 12]^2}\dd z \le \frac{4}{(4k+1)},
\end{align}
the combination of \eqref{eq:int_remainder} and \eqref{eq:int_contr} shows that 
\begin{align}\label{eq:rem_int}
    \abs{-\fint_{Q'_0}S(z)\dd z+ \sum_{y\in \Lambda\setminus \{0\},\abs{y}_\infty\le 2k} S'(y)-\sum_{y\in \Lambda\setminus \{0\},\abs{y}_\infty\le 2k} S(y)}\le \frac{4}{(4k+1)}.
\end{align}
We continue by estimating the parts of the sum in \eqref{eq:L} that satisfy $\abs{y}_\infty>2k$. 
Notice that for $z\in Q_y'$, it holds that $\abs{y-z}\le \sqrt{\frac 32} \le \frac 32$. Furthermore, the gradient of $S$ satisfies
\begin{align}
    \nabla S(z)=\frac{1}{\abs{z}^7}\begin{pmatrix}
    z_1(-3z_1^2+7z_2^2+2z_3^2)&\\
    z_2(-7z_1^2+3z_2^2-2z_3^2)&\\
    z_3(-5z_1^2+5z_2^2)&
    \end{pmatrix}
\end{align}
Using the estimate $\abs{S(y)-S(z)}\le \norm{\nabla S}_{L^\infty([y,z])}\abs{y-z}$, where $[y,z]=\set{\theta y+(1-\theta)z:\theta\in[0,1]}$ is the segment between $y$ and $z$, we infer for all $y \in \Lambda'$, $z \in Q_y'$
\begin{align*}
    \abs{S(y)-S(z)}\le \frac 32 \frac{7}{\abs{\abs{z}-\frac 32}^4}.
\end{align*}
We use this to estimate the sum outside a cube:
\begin{align}\label{eq:rem_sum}
\begin{aligned}
    \abs{\sum_{y\in \Lambda, \abs{y}_\infty>2k}S'(y)}&\le \frac{21}{2}\int_{ \abs{z}_\infty>2k-1}\frac{1}{\abs{\abs{z}-\frac 32}^4}\dd z
    \le \frac{21}{2}\int_{ \abs{z}>2k-1}\frac{1}{\abs{\abs{z}-\frac 32}^4}\dd z\\
    &=42\pi\int_{2k-1}^\infty \frac{r^2}{(r-\frac 32)^4}\dd r
    \le 42\pi \frac{(2k-1)^2}{(2k-\frac 52)^2}\int_{2k-1}^\infty \frac{1}{(r-\frac 32)^2}\dd r \\
    &=42\pi \frac{(2k-1)^2}{(2k-\frac 52)^3}
    =84\pi \frac{(4k-2)^2}{(4k-5)^3}\\
\end{aligned}
\end{align}
Combining \eqref{eq:rem_int} and \eqref{eq:rem_sum} yields 
\begin{align*}
    \abs{-\fint_{Q'_0}S(z)\dd z+ \sum_{y\in \Lambda\setminus \{0\}} S'(y)-\sum_{y\in \Lambda\setminus \{0\},\abs{y}_\infty\le 2k} S(y)}\le \frac{4}{(4k+1)}+84\pi \frac{(4k-2)^2}{(4k-5)^3}.
\end{align*}
The right hand side is smaller than $2.1$ for $k=35$ and thus the numerical result
\begin{align*}
   \sum_{y\in \Lambda\setminus \{0\},\abs{y}_\infty\le 70} S(y)\le -2.25
\end{align*}
shows \eqref{eq:L}. The numerical results were obtained with maple. For the source code we refer to the appendix at the end of the document.
\end{proof}